\documentclass[10pt,reqno]{amsart}
\usepackage{amsmath}
\usepackage{amssymb}
\usepackage{amsthm}
\usepackage{mathrsfs}

\newlength{\defbaselineskip}
\newcommand{\setlinespacing}[1]%
           {\setlength{\baselineskip}{#1 \defbaselineskip}}

\numberwithin{equation}{section}

\newtheorem{thm}{Theorem}[section]

\newtheorem{lem}[thm]{Lemma}
\newtheorem{prop}[thm]{Proposition}

\theoremstyle{definition}

\theoremstyle{remark}
\newtheorem{rem}[thm]{Remark}
\numberwithin{equation}{section}

\begin{document}
\title[Local smoothing and Strichartz estimates]{Local smoothing and Strichartz estimates for the Klein-Gordon equation with the inverse-square potential}

\author{Hyeongjin Lee, Ihyeok Seo and Jihyeon Seok}

\thanks{This research was supported by NRF-2019R1F1A1061316.}


\address{Department of Mathematics, Sungkyunkwan University, Suwon 16419, Republic of Korea}
\email{hjinlee@skku.edu}

\email{ihseo@skku.edu}

\email{jseok@skku.edu}

\begin{abstract}
We prove weighted $L^2$ estimates for the Klein-Gordon equation perturbed with singular potentials such as the inverse-square potential.
We then deduce the well-posedness of the Cauchy problem for this equation with small perturbations,
and go on to discuss local smoothing and Strichartz estimates which improve previously known ones. 	
\end{abstract}

\maketitle

\section{Introduction}		
Consider the Cauchy problem for the Klein-Gordon equation with a potential $V$:
\begin{equation}\label{KGE}
\begin{cases}
\partial_{t}^2 u-\Delta u+V(x)u+u = 0,\\
u(x,0)=f(x),\\
\partial_{t} u(x,0)=g(x),
\end{cases}
\end{equation}
where $(x,t)\in\mathbb{R}^{n}\times\mathbb{R}$ and $\Delta$ is the $n$ dimensional Laplacian.

In \cite{D'A}, D'Ancona established the following smoothing estimate (also known as \textit{local energy decay})
for the Klein-Gordon flow with $|V|\sim|x|^{-2}$
\begin{equation}\label{sd}
\big\||x|^{-1}e^{it\sqrt{-\Delta+V+1}}f\big\|_{L_{x,t}^2}\lesssim\|f\|_{H^{1/2}}
\end{equation}
by extending Kato's $H$-smoothing theory developed in \cite{K,KY} to the flow.
Recently, it was shown in \cite{D'A2} that the Klein-Gordon equation can have a solution with much smoothness locally as
\begin{equation}\label{LS0}
\sup_{R>0} \frac{1}{R} \int_{|x|<R} \int_{-\infty}^\infty
\big||\nabla|^{1/2}e^{it\sqrt{-\Delta+V+1}}f\big|^2 dtdx \lesssim \|f\|_{H^{1/2}}^2
\end{equation}
with $V$ such that for $\varepsilon>0$ and $\delta>0$
\begin{equation}\label{im}
|V(x)|\sim
\begin{cases}
|x|^{-2}|\log|x||^{-(1+\delta)}\quad\text{near the origin},\\
|x|^{-(2+\varepsilon)}|\log|x||^{-(1+\delta)}\quad\text{at infinity}.
\end{cases}
\end{equation}
At this point, it is worth noting that this local smoothing estimate can be written in terms of a weighted $L^2$ norm as in \eqref{sd}.
Indeed, if $\rho$  is any function such that $\sum_{j\in\mathbb{Z}}\|\rho\|_{L^\infty(|x|\sim2^j)}^2<\infty$,
one has
$$\|\rho|x|^{-1/2}v\|_{L_{x,t}^2}^2\lesssim\sup_{R>0} \frac{1}{R} \int_{|x|<R} \int_{-\infty}^\infty|v|^2 dtdx.$$
A typical example of such $\rho$ is given by
$$\rho=\sqrt{1+(\log|x|)^2}^{-(1/2+\varepsilon)},\quad \varepsilon>0,$$
as mentioned in \cite{D'A2}, and then \eqref{LS0} implies a weaker estimate,
\begin{equation*}
\|\rho|x|^{-1/2}|\nabla|^{1/2}e^{it\sqrt{-\Delta+V+1}}f\|_{L_{x,t}^2} \lesssim \|f\|_{H^{1/2}}.
\end{equation*}

Since the critical behavior for dispersion appears to be $|V|\sim|x|^{-2}$,
recent studies for perturbed dispersive equations have intensively aimed to get as close as possible to this inverse-square potential.
One of the main aims of this paper is to obtain the local smoothing estimate \eqref{LS0}
allowing small perturbations with the inverse-square potential that improves \eqref{im}.
More generally, we consider Fefferman-Phong potentials $V \in {\mathcal{F}}^p$ with small $\|V\|_{\mathcal{F}^{p}}$.
They are indeed defined for $1\leq p\leq n/2$ by
\begin{equation*}
\|V\|_{\mathcal{F}^p} = \sup_{x \in \mathbb{R}^n, r>0}r^{2-n/p}
\bigg(\int_{|y-x|<r} |V(y)|^p dy\bigg)^{\frac{1}{p}} < \infty.
\end{equation*}
It then follows that $L^{n/2} = \mathcal{F}^{n/2}$ and $|x|^{-2} \in L^{n/2, \infty} \subsetneq \mathcal{F}^p$ if $1 \le p < n/2$.
Let us now denote by $L^2(w)$ a weighted $L^2$ space equipped with the norm $\|f\|_{L^2(w)}=\big(\int|f(x)|^2w(x)dx\big)^{1/2}$.
Our main result is then the following theorem in which \eqref{sd} is generalized by \eqref{WL2} and \eqref{LS0} is improved by \eqref{LS}.

\begin{thm}\label{thm1}	
Let $n \ge 3$ and $V \in \mathcal{F}^{p}$ with $\|V\|_{\mathcal{F}^{p}}$ small enough for $p > (n-1)/2$.
Then there exists a unique solution $u\in L_{x,t}^2(|V|)$ to \eqref{KGE}
with Cauchy data $(f,g)\in H^{1/2}\times H^{-1/2}$ such that
\begin{equation}\label{cont2}
u\in C([0,\infty);H^{1/2}(\mathbb{R}^n))\quad\text{and}\quad
u_t \in C([0,\infty);H^{-1/2}(\mathbb{R}^{n})).
\end{equation}
Furthermore,
\begin{equation}\label{WL2}
\|u\|_{L_{x,t}^2 (|V|)} \lesssim \|V\|_{\mathcal{F}^{p}}^{1/2}\big(\|f\|_{H^{1/2}} + \|g\|_{H^{-1/2}}\big)
\end{equation}
and
\begin{equation}\label{LS}
\sup_{x_0\in\mathbb{R}^n,R>0} \frac{1}{R} \int_{|x-x_0|<R} \int_{-\infty}^\infty \big||\nabla|^{1/2}u\big|^2dtdx
\lesssim \|f\|_{H^{1/2}}^{2} + \|g\|_{H^{-1/2}}^{2}.
\end{equation}
\end{thm}

\begin{rem}\label{rem2}
Our theorem particularly considers perturbations of the type $a/|x|^2$ with a small $|a|>0$.
We would like to mention that this smallness condition is a natural restriction
when applying the quadratic form techniques to define the Schr\"odinger operator $-\Delta+a/|x|^2$ in some cases.
See \cite{RS}, p. 172.
\end{rem}

\begin{rem}
The estimate \eqref{LS} shows that the energy in a cylinder $\{(x,t)\in\mathbb{R}^n\times\mathbb{R}:|x-x_0|<R\}$
decays as the square root of the radius $R$. In this regard, it may be also seen as a local energy decay.
\end{rem}

One of the key ingredients in the proof of Theorem \ref{thm1} is weighted $L^2$ estimates for solutions of the inhomogeneous Klein-Gordon
equation $\partial_{t}^2 u-\Delta u+u =F(x,t)$, which are obtained in Sections \ref{sec2} and \ref{sec3}.
Making use of these estimates, we also obtain the following Strichartz estimate for the perturbed Klein-Gordon equation \eqref{KGE}.
\begin{thm}\label{thm2}
Let $n \ge 3$. Assume that $u$ is a solution of \eqref{KGE} with Cauchy data $(f, g) \in H^{1/2}\times H^{-1/2}$ and potential
$V \in \mathcal{F}^p$ with small $\|V\|_{\mathcal{F}^p}$ for $p > (n-1)/2$.
Then we have
	\begin{equation}\label{Strichartz}
	\|u\|_{L_t^q(\mathbb{R};H_r^\sigma(\mathbb{R}^n))} \lesssim (1 + \|V\|_{\mathcal{F}^p}) \big( \|f\|_{H^{1/2}} + \|g\|_{H^{-1/2}} \big)
	\end{equation}
where the pairs $q>2$ and $r\geq2$ satisfy the admissible condition for $0 \le \theta \le 1$
	\begin{equation}\label{admissible0}
	\frac{2}{q} + \frac{n-1+\theta}{r} \le \frac{n-1+\theta}{2},
	\end{equation}
and $\sigma\geq0$ satisfies the gap condition
\begin{equation}\label{admissible}
\sigma = \frac{1}{q} + \frac{n+\theta}{r} -\frac{n-1+\theta}{2}.
\end{equation}
\end{thm}

\begin{rem}\label{rem}
The condition \eqref{admissible0} corresponds to the wave admissible pairs at $\theta=0$
and the equality in \eqref{admissible0} is the Schr\"odinger admissible pairs at $\theta=1$.
Note that the Klein-Gordon flow $e^{it\sqrt{1-\Delta}}$ behaves like the wave flow at high frequency and the Schr\"odinger flow at low frequency.
In general, the Strichartz estimates for the Klein-Gordon equation are more complicated by the different scaling of $\sqrt{1-\Delta}$
for low and high frequencies.
\end{rem}

We conclude with a short summary of earlier results.
The Strichartz estimates for the Klein-Gordon equation have been studied for decades.

In the free case $V\equiv 0$, Strichartz \cite{St} first established the following estimate
in connection with the Fourier restriction theory in harmonic analysis:
\begin{equation}\label{one-s}
\|u\|_{L^q(\mathbb{R}^{n+1})} \lesssim \|f\|_{H^{1/2}} + \|g\|_{H^{-1/2}}
\end{equation}
for $2(n+2)/n \le q \le 2(n+1)/(n-1)$.
Since then, there have been developments in extending \eqref{one-s} to mixed norm spaces $L_t^q(\mathbb{R}; L_x^r(\mathbb{R}^n))$
as follows (see e.g. \cite{B,ZZ} and references therein):
\begin{equation}\label{tre}
\|u\|_{L_t^q(\mathbb{R};H_r^\sigma(\mathbb{R}^n))} \lesssim \|f\|_{H^{1/2}} + \|g\|_{H^{-1/2}}
\end{equation}
under the same conditions \eqref{admissible0} and \eqref{admissible}.
Here the diagonal case $q=r$ with $\sigma=0$ entirely recovers \eqref{one-s}.

In the case of the potential perturbation, several works have recently treated \eqref{Strichartz}.
In \cite{D'AF} the potentials satisfy the decay assumption that $V(x)$ decays like $(|x|^{3/2-\varepsilon}+|x|^2)^{-1}$ at infinity
only with the non-endpoint ($q\neq2$) Schr\"odinger admissible pairs. See Remark \ref{rem}.
In \cite{D'A2} this is extended to the wave admissible pairs but with the stronger assumption \eqref{im} on the potential.
Compared with these previous results, our theorem improves not only the perturbation by the inverse-square potential $c/|x|^{-2}$
but also the pairs $(q,r)$ for which the estimate holds.

The rest of this paper is organized as follows. In Sections \ref{sec2} and \ref{sec3},
we obtain some weighted-$L^2$ and local smoothing estimates concerning the free Klein-Gordon flow $e^{it\sqrt{1-\Delta}}$,
which will be used in the next sections \ref{sec4} and \ref{sec5}
for the proof of Theorems \ref{thm1} and \ref{thm2}, respectively.

Throughout this paper, the letter $C$ stands for a positive constant which may be different
at each occurrence.
We also denote $A\lesssim B$ to mean $A\leq CB$
with unspecified constants $C>0$.

\section{Estimates for the free flow}\label{sec2}

In this section we obtain some estimates for the free Klein-Gordon flow $e^{it\sqrt{1-\Delta}}$
which will be used in the next sections for the proof of Theorems \ref{thm1} and \ref{thm2}.
From now on, we shall use the notation $\langle \nabla\rangle=\sqrt{1-\Delta}$.

\begin{prop}\label{KGWL2}
Let $n\geq3$ and $V\in \mathcal{F}^{p}$ for $p>(n-1)/2$. Then we have
\begin{equation}\label{homo}
\big\| e^{it\sqrt{1-\Delta}}f\big\|_{L_{x,t}^2 (|V|)} \lesssim \|V\|_{\mathcal{F}^{p}}^{1/2} \|\langle \nabla \rangle^{1/2}f\|_{L^2}
\end{equation}
and
\begin{equation}\label{inhomo}
\sup_{t\in \mathbb{R}}\bigg\|\langle \nabla\rangle^{-1/2} \int_{0}^{t} e^{i(t-s)\sqrt{1-\Delta}}F(\cdot,s)ds \bigg\|_{L_x^2} \lesssim \|V\|_{\mathcal{F}^{p}}^{1/2} \|F\|_{L_{x,t}^2(|V|^{-1})}.
\end{equation}
\end{prop}

\begin{proof}
First we show the estimate \eqref{homo}.
Using polar coordinates $\xi \rightarrow r\sigma$ and a change of variables $\sqrt{1+r^2} \rightarrow r$, we see
\begin{align*}
e^{it\sqrt{1-\Delta}}f(x) &=  \int_{\mathbb{R}^n} e^{ix\cdot\xi} e^{it\sqrt{1+|\xi|^{2}}} \widehat{f}(\xi)\,d\xi\\
&= \int_{0}^{\infty} \int_{\mathbb{S}_{r}^{n-1}} e^{ix\cdot r\sigma} e^{it\sqrt{1+r^{2}}} \widehat{f}(r\sigma) \,d\sigma_{r}dr\\
&=\int_{1}^{\infty} \int_{\mathbb{S}_{\sqrt{r^2-1}}^{n-1}} e^{ix\cdot\sqrt{r^2-1}\sigma} e^{itr}\widehat{f}(\sqrt{r^2-1}\sigma)
\frac{r}{\sqrt{r^2-1}} \,d\sigma_{\sqrt{r^2-1}}\,dr\\
&=\int_{-\infty}^\infty e^{itr}\chi_{(1,\infty)}(r) \frac{r}{\sqrt{r^2-1}} \big(\widehat{f}d\sigma_{\sqrt{r^2-1}}\big)^{\wedge}(-x)\,dr.
\end{align*}
By applying Plancherel's theorem in the $t$-variable and then using a change of variables $ \sqrt{r^2-1} \rightarrow r$ again,
\begin{align}\label{homopf1}
\big\| e^{it\sqrt{1-\Delta}}f \big\|_{L_{x,t}^{2}(|V|)}^{2} &=\int_{\mathbb{R}^n}\int_{1}^{\infty}
\bigg|\frac{r}{\sqrt{r^2-1}} \big(\widehat{f}d\sigma_{\sqrt{r^2-1}}\big)^{\wedge}(-x)\bigg|^{2}|V(x)|\,drdx\nonumber\\
&=\int_{0}^{\infty}  \int_{\mathbb{R}^n}\frac{\sqrt{1+r^2}}{r}\Big|\widehat{\widehat{f}d\sigma_{r}}(-x)\Big|^{2}|V(x)|dxdr\nonumber\\
&=\int_{0}^{\infty} \frac{\sqrt{1+r^2}}{r}\Big\| \widehat{\widehat{f}d\sigma_{r}}(-\cdot)\Big\|_{L_x^2(|V|)}^2dr.
\end{align}
Now we make use of the following weighted $L^2$ restriction estimate for the Fourier transform:
\begin{equation}\label{weightfr}
\big\|\widehat{fd\sigma}\big\|_{L^{2}(|V|)} \lesssim \|V\|_{\mathcal{F}^{p}}^{1/2}\|f\|_{L^{2}(\mathbb{S}^{n-1})}
\end{equation}
for $V\in \mathcal{F}^p$ with $p>(n-1)/2$, $n\geq3$, which can be found in \cite{CR,CS} (see also \cite{BBCRV,S}).
Indeed, applying the re-scaled estimate of \eqref{weightfr},
\begin{equation*}
\big\|\widehat{fd\sigma_{r}}\big\|_{L^{2}(|V|)} \lesssim r^{1/2}\|V\|_{\mathcal{F}^{p}}^{1/2}\|f\|_{L^{2}(\mathbb{S}_{r}^{n-1})},
\end{equation*}
to the right-hand side of \eqref{homopf1}, we get
\begin{align*}
\big\|e^{it\sqrt{1-\Delta}}f\big\|_{L_{x,t}^2(|V|)}^{2}
&\lesssim \int_{0}^{\infty} \sqrt{1+r^2} \|V\|_{\mathcal{F}^p}\int_{\mathbb{S}_r^{n-1}} |\widehat{f}(r\sigma)|^{2}d\sigma_{r}dr\\
&= \|V\|_{\mathcal{F}^p} \int_{0}^{\infty}\int_{\mathbb{S}_r^{n-1}}|(1+r^2)^{1/4} \widehat{f}(r\sigma)|^{2}d\sigma_r dr\\
&= \|V\|_{\mathcal{F}^p} \|\langle \nabla \rangle^{1/2}f\|_{L^2}^2
\end{align*}
as desired.

To show the second estimate \eqref{inhomo}, we first note that \eqref{homo} is equivalent to
\begin{equation}\label{eq2}
\bigg\|\langle \nabla \rangle^{-1/2} \int_{-\infty}^\infty e^{-is\sqrt{1-\Delta}}F(\cdot,s)ds\bigg\|_{L_x^2} \lesssim \|V\|_{\mathcal{F}^p}^{1/2}\|F\|_{L_{x,t}^2 (|V|^{-1})}
\end{equation}
by duality.
Substituting $\chi_{[0,t]}(s)F(\cdot,s)$ for $F(\cdot,s)$ and taking the supremum over $t$, we then get
\begin{equation*}
\sup_{t\in\mathbb{R}} \bigg\| \langle \nabla \rangle^{-1/2} \int_{0}^{t} e^{-is\sqrt{1-\Delta}}F(\cdot,s)ds\bigg\|_{L_x^2}
\lesssim \|V\|_{\mathcal{F}^{p}}^{1/2} \|F\|_{L_{x,t}^{2}(|V|^{-1})}.
\end{equation*}
Combining this and the fact that $e^{it\sqrt{1-\Delta}}$ is an isometry in $L^2$,
we obtain the desired estimate,
\begin{align*}
\sup_{t\in \mathbb{R}}\bigg\|\langle \nabla \rangle^{-1/2}\int_{0}^{t}e^{i(t-s)\sqrt{1-\Delta}}F(\cdot,s)\,ds\bigg\|_{L_x^2}
&= \sup_{t\in \mathbb{R}}\bigg\| \langle \nabla \rangle^{-1/2}\int_{0}^{t}e^{-is\sqrt{1-\Delta}}F(\cdot,s)\,ds\bigg\|_{L_x^2}\\
&\lesssim \|V\|_{\mathcal{F}^{p}}^{1/2} \|F\|_{L_{x,t}^2(|V|^{-1})}.
\end{align*}
\end{proof}

Next we obtain the following local smoothing estimate for the free flow
using Lemma \ref{tracethm} instead of \eqref{weightfr} in the arguments of the proof of \eqref{homo}.

\begin{prop}\label{LS-homo}
Let $n \geq 2$. Then we have
\begin{equation}\label{LShomo}
\sup_{x_0 \in \mathbb{R}^n ,R>0} \frac{1}{R} \int_{|x-x_0|<R} \int_{-\infty}^{\infty} \big| |\nabla|^{1/2} e^{it\sqrt{1-\Delta}} f \big|^{2}\,dtdx
\lesssim \| f \|^2_{H^{1/2}}.
\end{equation}
\end{prop}

\begin{proof}
Using polar coordinates $\xi \rightarrow r\sigma$ and a change of variables $\sqrt{1+r^2} \rightarrow r$ as before, we see
\begin{align*}
|\nabla|^{\frac12} e^{it\sqrt{1-\Delta}}f(x)
&=\int_{0}^{\infty} e^{it\sqrt{1+r^{2}}}r^{\frac12} \int_{\mathbb{S}_{r}^{n-1}} e^{ix\cdot r\sigma}\widehat{f}(r\sigma)d\sigma_{r} dr\\
&=\int_{-\infty}^\infty e^{itr}\chi_{(1,\infty)}(r) \frac{r}{(r^2-1)^{\frac14}} \big(\widehat{f}d\sigma_{\sqrt{r^2-1}}\big)^{\wedge}(-x)dr.
\end{align*}
We then apply Plancherel's theorem in the $t$-variable and use a change of variables $\sqrt{r^2-1}\rightarrow r$ to obtain
\begin{align}\label{eq3}
\nonumber\sup_{x_0 ,R} \frac{1}{R} \int_{|x-x_0|<R} &\int_{-\infty}^{\infty}  \big| |\nabla|^{1/2} e^{it\sqrt{1-\Delta}} f \big|^{2}dtdx\\
&=\sup_{x_0 ,R}\frac{1}{R} \int_{|x-x_0|<R}\int_{1}^{\infty}\bigg|\frac{r}{(r^2-1)^{1/4}} \nonumber\big(\widehat{f}d\sigma_{\sqrt{r^2-1}}\big)^{\wedge}(-x)\bigg|^{2}drdx\\
&=\sup_{x_0,R} \frac{1}{R} \int_{|x-x_0|<R}\int_{0}^{\infty} \sqrt{1+r^2} \Big| \widehat{\widehat{f}d\sigma_r}(-x)\Big|^{2}drdx.
\end{align}
At this point we use an elementary result for Fourier transforms of $L^2$ densities
(see Theorem 2.1 in \cite{AH}):
\begin{lem}\label{tracethm}
Let $K$ be a compact subset of a $C^{1}$ manifold $M$ of codimension $k$ in $\mathbb{R}^n$,
with the Euclidean surface area $dS$.
If the Fourier transform $\widehat{g}$ of a tempered distribution $g\in\mathcal{S}'(\mathbb{R}^n)$ is a square integrable density $\widehat{g}dS$
with support in $K$, then
\begin{equation}\label{den}
\int_{|x|<R} |g(x)|^{2}dx \leq CR^{k} \int |\widehat{g}(\xi)|^{2}dS
\end{equation}
where $\xi=(\xi_1,\cdots, \xi_{n-k}) \in \mathbb{R}^{n-k}$ and the constant $C$ is independent of $g$ and $R>0$.	
\end{lem}

Indeed, by changing the order of integration in the right side of \eqref{eq3} and then applying \eqref{den} with $k=1$ and $g=\widehat{\widehat{f}d\sigma_r}$,
it follows that
\begin{align*}\label{eq3}
\sup_{x_0 ,R} \frac{1}{R} \int_{|x-x_0|<R}\int_{-\infty}^{\infty}\big| |\nabla|^{\frac12} e^{it\sqrt{1-\Delta}} f \big|^{2}dtdx
&\lesssim \sup_{x_0,R}\int_{0}^{\infty} \sqrt{1+r^2} \int_{\mathbb{S}_{r}^{n-1}} |\widehat{f}(r\sigma)|^2 d\sigma_{r}dr\\
&\lesssim \int_{0}^{\infty}  \int_{\mathbb{S}_{r}^{n-1}} |(1+r^2)^{1/4}\widehat{f}(r\sigma)|^{2}d\sigma_{r}dr \\
&=\|f\|^2_{H^{1/2}}
\end{align*}
as desired.
\end{proof}

\section{Inhomogeneous estimates}\label{sec3}
In addition to the estimates in the previous section, we need to obtain the corresponding estimates for the inhomogeneous Klein-Gordon equation
with zero initial data,
\begin{equation}
	\begin{cases}\label{inhomoKGE}
\partial_t^2 u - \Delta u + u = F(x,t),\\
u(x,0)=0,\\
\partial_t u(x,0)=0,
	\end{cases}
\end{equation}
as follows:
\begin{prop}\label{LS-inhomo}
Let $n \geq 3$ and $V \in \mathcal{F}^{p}$ for $p>(n-1)/2$. If $u$ is a solution of \eqref{inhomoKGE}, then
\begin{equation}\label{WLinhomo}
\| u \|_{L_{x,t}^{2}(|V|)} \lesssim \| V \|_{\mathcal{F}^{p}} \| F \|_{L_{x,t}^2(|V|^{-1})}
\end{equation}
and
\begin{equation}\label{LSinhomo}
\sup_{x_{0} \in \mathbb{R}^n,R>0} \frac{1}{R} \int_{|x-x_0|<R}\int_{-\infty}^{\infty} \big||\nabla|^{1/2} \, u\big|^2 \,dtdx \lesssim \|V\|_{\mathcal{F}^{p}}\| F \|_{L_{x,t}^2(|V|^{-1})}^2.
\end{equation}
\end{prop}

The main ingredient in the proof of this proposition is the following estimates for the resolvent $\mathcal{R}(z):=(-\Delta-z)^{-1}$
of the Laplacian.
The first estimate \eqref{WL2-helm} can be found in \cite{CS,CR} (see also \cite{S}),
and see Theorem 2 in \cite{RV} for the second estimate \eqref{LS-helm}.

\begin{lem}\label{Helm}		
Let $n\geq 3$ and $z\in\mathbb{C}$ with $\textrm{Im}\,z\neq0$. If $V\in\mathcal{F}^{p}$ for $p> (n-1)/2$, then
\begin{equation}\label{WL2-helm}
\|\mathcal{R}(z)f\|_{L^2(|V|)}\leq C\|V\|_{\mathcal{F}^{p}} \|f\|_{L^2(|V|^{-1})}
\end{equation}
and
\begin{equation}\label{LS-helm}
\sup_{x_{0} \in \mathbb{R}^n,R>0} \frac{1}{R} \int_{|x-x_0|<R} \big||\nabla|^{1/2}\mathcal{R}(z)f \big|^2 dx \leq C \|V\|_{\mathcal{F}^{p}} \|f\|_{L^2(|V|^{-1})}^2
\end{equation}
with constants $C>0$ independent of $z$.
\end{lem}

\subsection{Proof of \eqref{WLinhomo}}
We first decompose the solution $u$ of \eqref{inhomoKGE} as
\begin{equation}\label{decom}
u(x,t)=\tilde{u}(x,t)+R(x,t),
\end{equation}
where
\begin{equation*}
\tilde{u}(x,t)=\lim_{\varepsilon \rightarrow 0^+} \int_{-\infty}^{\infty}\int_{\mathbb{R}^n} e^{ix\cdot \xi + it\tau} \frac {\widehat{F}(\xi,\tau)}{1+|\xi|^2 - (\tau+i\varepsilon)^2}\,d\xi d\tau
\end{equation*}
and
\begin{align*}
R(x,t)=&-\cos{(t\sqrt{1-\Delta})} \int_{-\infty}^{\infty} \frac{\sin{(t\sqrt{1-\Delta})}}{\sqrt{1-\Delta}}\big(\chi_{[0,\infty)}(t)F(\cdot,t)\big)dt \\
& -i \frac{\sin{(t\sqrt{1-\Delta})}}{\sqrt{1-\Delta}}\int_{-\infty}^{\infty}\cos{\big(t\sqrt{1-\Delta}\big)}\big(\chi_{[0,\infty)}(t) F(\cdot,t)\big) dt .
\end{align*}
This decomposition enables us to control the solution $u$ by dividing it into two parts.
We handle the main part $\widetilde{u}$ appealing to the resolvent estimate in Lemma \ref{Helm}
while applying the weighted $L^2$ estimates for the Klein-Gordon flow $e^{it\sqrt{1-\Delta}}$ to the remainder part $R$.
Assuming for the moment this decomposition, we will show
\begin{equation}\label{md1}
\| \tilde{u} \|_{L_{x,t}^{2}(|V|)} \lesssim \| V \|_{\mathcal{F}^{p}} \| F \|_{L_{x,t}^2(|V|^{-1})}
\end{equation}
and
\begin{equation}\label{rd1}
\| R \|_{L_{x,t}^{2}(|V|)} \lesssim \| V \|_{\mathcal{F}^{p}} \| F \|_{L_{x,t}^2(|V|^{-1})}
\end{equation}
which implies immediately the desired estimate \eqref{WLinhomo}.
The second estimate \eqref{rd1} follows easily by applying the weighted $L^2$ estimates \eqref{homo} and \eqref{eq2}.
On the other hand, for the first estimate \eqref{md1} we use Plancherel's theorem in the $t$-variable and then change the order of integration to obtain
\begin{align}\label{md1-1}
\nonumber\| \tilde{u} \|_{L_{x,t}^{2}(|V|)}^2
&\lesssim \lim_{\varepsilon \rightarrow 0^+} \bigg\| \int_{-\infty}^{\infty}\int_{\mathbb{R}^n} e^{ix\cdot \xi + it\tau} \frac {\widehat{F}(\xi,\tau)} {1+|\xi|^{2} - (\tau+i\epsilon)^{2}} d\xi d\tau\bigg\|_{L_{x,t}^2 (|V|)}^2\\
&=\lim_{\varepsilon \rightarrow 0^+} \int_{-\infty}^{\infty} \int_{\mathbb{R}^n}\bigg|\int_{\mathbb{R}^n} e^{ix\cdot\xi}\frac{\widehat{F}(\xi,\tau)}{1+|\xi|^2 - (\tau+i\epsilon)^2 }d\xi\bigg|^2 |V(x)|dx d\tau .
\end{align}
Applying the resolvent estimate \eqref{WL2-helm} and Plancherel's theorem in the $\tau$-variable to \eqref{md1-1}, we have
\begin{align*}
\| \tilde{u} \|_{L_{x,t}^{2}(|V|)}^2
&\lesssim \int_{-\infty}^{\infty} \| V \|_{\mathcal{F}^p}^2 \int_{\mathbb{R}^n} |\widehat{F(x,\cdot)}(\tau)|^{2}|V(x)|^{-1}dxd\tau\\
&=\| V \|_{\mathcal{F}^p}^2 \int_{\mathbb{R}^n}\int_{-\infty}^{\infty}  |F(x,t)|^2 |V(x)|^{-1}dtdx\\
&=\| V \|_{\mathcal{F}^p}^2 \| F \| _{L_{x,t}^{2}(|V|^{-1})}^{2}
\end{align*}
as desired.

Now it remains to prove \eqref{decom}.
Since $\tilde{u}=u-R$ is a solution to the inhomogeneous equation via the space-time Fourier transform,
the remainder term $R$ is the solution of the homogeneous problem:
\begin{equation*}
\begin{cases}
\partial_t^2 R - \Delta R + R = 0,\\
R(x,0)=-\tilde{u}(x,0),\\
\partial_t R(x,0)=-\partial_t \tilde{u}(x,0)
\end{cases}
\end{equation*}
whose solution is given by
\begin{equation}\label{090}
R(x,t)=\cos{(t\sqrt{1-\Delta})}\big(-\tilde{u}(\cdot,0)\big)
+\frac{\sin{(t\sqrt{1-\Delta})}}{\sqrt{1-\Delta}}\big(-\partial_t \tilde{u}(\cdot,0)\big).
\end{equation}
To get $R$ explicitly, we compute
\begin{align*}
\tilde{u}(x,0)&=\lim_{\varepsilon \rightarrow 0^+}\int_{-\infty}^{\infty}\int_{\mathbb{R}^n} e^{ix\cdot \xi} \frac {\widehat{F}(\xi,\tau)}{1+|\xi|^2 - (\tau+i\varepsilon)^2}\,d\xi d\tau \\
&=\lim_{\varepsilon \rightarrow 0^+} \int_{-\infty}^{\infty}\int_{\mathbb{R}^{n}} e^{ix\cdot\xi}
\int_{-\infty}^{\infty} \frac{e^{-it\tau}}{1+|\xi|^2 - (\tau+i\varepsilon)^2} \widehat{F(\cdot,t)}(\xi) \,dtd\xi d\tau\\
&=\int_{\mathbb{R}^n}\int_{-\infty}^{\infty} e^{ix\cdot \xi} \widehat{F(\cdot,t)}(\xi)
\bigg(\lim_{\varepsilon \rightarrow 0^+}\int_{-\infty}^{\infty} \frac{e^{-it\tau}}{1+|\xi|^2-(\tau+i\varepsilon)^2}\,d\tau\bigg) dt d\xi
\end{align*}
and note that
\begin{align*}
&\lim_{\varepsilon \rightarrow 0^+}\int_{-\infty}^{\infty} \frac{e^{-it\tau}}{1+|\xi|^2-(\tau+i\varepsilon)^2}\,d\tau\\
&=\frac{1}{2\sqrt{1+|\xi|^2}} \Bigg\{ \lim_{\varepsilon \rightarrow 0^+}\int_{-\infty}^{\infty} \frac{e^{-it\tau}d\tau}{\sqrt{1+|\xi|^2}+\tau+i\varepsilon}
+ \lim_{\varepsilon \rightarrow 0^+}\int_{-\infty}^{\infty} \frac{e^{-it\tau}d\tau}{\sqrt{1+|\xi|^2}-(\tau+i\varepsilon)} \Bigg\}\\
&=\frac{1}{2\sqrt{1+|\xi|^2}} \Big\{ e^{it\sqrt{1+|\xi|^2}}(-i\chi_{[0,\infty)}(t))-e^{-it\sqrt{1+|\xi|^2}}(-i\chi_{[0,\infty)}(t)) \Big\}\\
&=\frac{\sin (t\sqrt{1+|\xi|^2})}{\sqrt{1+|\xi|^2}}\chi_{[0,\infty)}(t)
\end{align*}
(see \cite{So}, p. 30).
Therefore, we get
\begin{align}\label{091}
\nonumber\tilde{u}(x,0)
&=\int_{\mathbb{R}^n}\int_{-\infty}^{\infty} e^{ix\cdot\xi}\frac{\sin{(t\sqrt{1+|\xi|^2})}}{\sqrt{1+|\xi|^2}} \widehat{F(\cdot,t)}(\xi)\chi_{[0,\infty)}(t)dt d\xi\\
&=\int_{-\infty}^{\infty} \frac{\sin{(t\sqrt{1-\Delta})}}{\sqrt{1-\Delta}} \big(\chi_{[0,\infty)}(t) F(\cdot,t)\big)dt.
\end{align}
Similar calculations give
\begin{equation*}
\partial_t \tilde{u}(x,0)
=i\int_{-\infty}^{\infty}\cos{(t\sqrt{1-\Delta})}\big(\chi_{[0,\infty)}(t) F(\cdot,t)\big)dt.
\end{equation*}
Inserting this and \eqref{091} into \eqref{090}, we get the decomposition \eqref{decom}.

\subsection{Proof of \eqref{LSinhomo}}
Next we prove \eqref{LSinhomo}.
By the decomposition \eqref{decom}, it is enough to show that both $\tilde{u}$ and $R$ satisfy the estimate \eqref{LSinhomo}.
For $R$, it is easily shown by applying \eqref{LShomo} and \eqref{eq2}.
On the other hand, for $\tilde{u}$ we use Plancherel's theorem in the $t$-variable and then change the order of integration to get
\begin{align}\label{md2}
&\sup_{x_{0},R} \frac{1}{R} \int_{|x-x_0|<R} \int_{-\infty}^{\infty} \big||\nabla|^{1/2} \tilde{u}(x,t)\big|^{2}dtdx \nonumber\\
&=\sup_{x_{0},R} \frac{1}{R} \int_{|x-x_0|<R}\int_{-\infty}^{\infty} \bigg|\int_{-\infty}^{\infty}\int_{\mathbb{R}^n} e^{ix\cdot \xi + it\tau} \frac {|\xi|^{1/2}\widehat{F}(\xi,\tau)} {1+|\xi|^{2} - (\tau+i\epsilon)^{2}} \,d\xi d\tau\bigg|^2 dtdx\nonumber\\
&\lesssim \int_{-\infty}^{\infty}\, \sup_{x_{0},R} \frac{1}{R} \int_{|x-x_0|<R}\bigg|\int_{\mathbb{R}^n} e^{ix\cdot \xi} \frac {|\xi|^{1/2}\widehat{F}(\xi,\tau)} {1+|\xi|^{2} - (\tau+i\epsilon)^{2}} \,d\xi \bigg|^{2} dx d\tau.
\end{align}
Using \eqref{LS-helm} and then applying Plancherel's theorem in the $\tau$-variable again, the right-hand side of \eqref{md2} is bounded by
\begin{align*}
\| V \|_{\mathcal{F}^p} \int_{-\infty}^{\infty}  \int_{\mathbb{R}^{n}} |\widehat{F(x,\cdot)}(\tau)|^2 |V(x)|^{-1}dxd\tau
&=\|V\|_{\mathcal{F}^p} \int_{\mathbb{R}^n}\int_{-\infty}^{\infty}  |F(x,t)|^{2} |V(x)|^{-1}dtdx \\
&=\| V \|_{\mathcal{F}^p} \|F\|_{L_{x,t}^2(|V|^{-1})}^{2}.
\end{align*}
Hence we get
\begin{equation*}
\sup_{x_{0},R} \frac{1}{R} \int_{|x-x_0|<R} \int_{-\infty}^{\infty} \big||\nabla|^{\frac{1}{2}} \tilde{u}(x,t)\big|^{2}dtdx
\lesssim\| V \|_{\mathcal{F}^p} \|F\|_{L_{x,t}^2(|V|^{-1})}^{2}
\end{equation*}
as desired.

\section{Proof of Theorem \ref{thm1}}\label{sec4}
This section is devoted to proving Theorem \ref{thm1}.
Let us first define the space $X$ and the operator $\mathscr{S}$ by
$$X = \{ F: \| F \|_{L_{x,t}^2 (|V|)} < \infty\}$$
and
$$\mathscr{S}F = \int_{0}^{t}\frac{\sin((t-s)\sqrt{1-\Delta})}{\sqrt{1-\Delta}}\big(V(\cdot)F(\cdot,s)\big)ds.$$
We then consider the potential term in \eqref{KGE} as a source term and thus write the solution of \eqref{KGE} as
the sum of the solution to the free Klein-Gordon equation plus a Duhamel term, as follows:
\begin{equation}\label{KGSOL}
u=\cos(t\sqrt{1-\Delta}) f + \frac {\sin(t\sqrt{1-\Delta})}{\sqrt{1-\Delta}} g + \mathscr{S}u.
\end{equation}

For the existence of a unique solution in the space $X$, we will show that the first two terms in the right-hand side of \eqref{KGSOL} are in $X$,
provided $(f,g)\in H^{1/2}\times H^{-1/2}$,
and that the operator $\mathscr{S}$ is a contraction in $X$ if $\| V \|_{\mathcal{F}^p}$ is small enough.
The first part follows immediately from applying the estimate \eqref{homo}:
\begin{align}\label{WL2homo}
\bigg\| \cos(t\sqrt{1-\Delta}) f + \frac {\sin(t\sqrt{1-\Delta})}{\sqrt{1-\Delta}} g \bigg\|_{L_{x,t}^2(|V|)}
\lesssim \| V \|_{\mathcal{F}^p}^{1/2}\big(\| f \|_{H^{1/2}} + \| g \|_{H^{-1/2}}\big).
\end{align}
For the second part, by the estimate \eqref{WLinhomo} we see that
\begin{align}\label{WL2inhomo}
\|\mathscr{S}F \|_{L_{x,t}^{2}(|V|)}
\nonumber&\lesssim \| V \|_{\mathcal{F}^p} \| VF \|_{L_{x,t}^2(|V|^{-1})}\\
\nonumber&\lesssim \| V \|_{\mathcal{F}^p} \| F\|_{L_{x,t}^2(|V|)}\\
&\leq\frac12\| F\|_{L_{x,t}^2(|V|)}
\end{align}
for $F\in X$, and thus the operator $\mathscr{S}$ is a contraction in $X$
since we are assuming that $\| V \|_{\mathcal{F}^p}$ is small enough.

Next, we show \eqref{cont2}. First note that
\begin{equation*}
\sup_{t\in\mathbb{R}}\bigg\| \cos(t\sqrt{1-\Delta})f + \frac{\sin(t\sqrt{1-\Delta})}{\sqrt{1-\Delta}}g\bigg\|_{H^{1/2}} \\
\lesssim \| \langle  \nabla \rangle^{1/2} f \|_{L^2} +\| \langle \nabla \rangle^{-1/2} g \|_{L^2}
\end{equation*}
by Plancherel's theorem,
and applying \eqref{inhomo} with $F=Vu$ gives
\begin{equation*}
\sup_{t\in\mathbb{R}}\| \mathscr{S}u \|_{H^{1/2}}
\lesssim \| V \|_{\mathcal{F}^p}^{1/2} \| u\|_{L_{x,t}^2(|V|)}.
\end{equation*}
It then follows easily that $u\in C([0,\infty);H^{1/2}(\mathbb{R}^{n}))$.
Similarly as above, to show $u_t\in C([0,\infty);H^{-1/2}(\mathbb{R}^{n}))$,
we first calculate
\begin{equation*}
\partial_t \bigg(\cos(t\sqrt{1-\Delta}) f + \frac{\sin(t\sqrt{1-\Delta})}{\sqrt{1-\Delta}} g \bigg)
= -\langle  \nabla\rangle \sin(t\sqrt{1-\Delta}) f + \cos(t\sqrt{1-\Delta}) g
\end{equation*}
and
\begin{equation*}
\frac{d}{dt} \int_{0}^{t} \frac{\sin((t-s)\sqrt{1-\Delta})} {\sqrt{1-\Delta}}F(\cdot,s)ds = \int_{0}^{t}\cos((t-s)\sqrt{1-\Delta})F(\cdot,s)ds.
\end{equation*}
Then we see
\begin{equation*}
\sup_{t\in\mathbb{R}}\bigg \|\partial_t \bigg(\cos(t\sqrt{1-\Delta}) f + \frac{\sin(t\sqrt{1-\Delta})}{\sqrt{1-\Delta}} g \bigg)\bigg\|_{H^{-1/2}} \lesssim \| f \|_{H^{1/2}} + \| g \|_{H^{-1/2}}
\end{equation*}
by Plancherel's theorem, and see
\begin{align*}
\sup_{t\in\mathbb{R}}\bigg\|\frac{d}{dt} \int_{0}^{t} \frac{\sin((t-s)\sqrt{1-\Delta})} {\sqrt{1-\Delta}}F(\cdot,s)ds\bigg\|_{H^{-1/2}}
&\lesssim \| V \|_{\mathcal{F}^{p}}^{1/2} \| u \| _{L_{x,t}^2(|V|)}.
\end{align*}
by \eqref{inhomo}.
It then follows from \eqref{KGSOL} that $u_t\in C([0,\infty);H^{-1/2}(\mathbb{R}^{n}))$.

Combining \eqref{KGSOL}, \eqref{WL2homo} and \eqref{WL2inhomo} directly yields the desired estimate \eqref{WL2}.
It remains only to show the local smoothing estimate \eqref{LS}.
But this follows immediately from applying \eqref{KGSOL}, \eqref{LShomo}, \eqref{LSinhomo}, and then \eqref{WL2}.

\section{Proof of Theorem \ref{thm2}}\label{sec5}

In this final section we prove Theorem \ref{thm2} by making use of the weighted $L^2$ estimates obtained in Sections \ref{sec2} and \ref{sec3}.
Recalling \eqref{KGSOL}, we first write the solution $u$ of \eqref{KGE} as
\begin{equation}\label{sol}
u= \cos(t\sqrt{1-\Delta})f + \frac{\sin(t\sqrt{1-\Delta})}{\sqrt{1-\Delta}}g + \mathscr{S}u.
\end{equation}
Applying the following Strichartz estimate for the free case (see \eqref{tre}),
\begin{equation}\label{Shomo}
\|e^{it\sqrt{1-\Delta}}f\|_{L_t^q H_r^\sigma} \lesssim \|f\|_{H^{1/2}},
\end{equation}
to the first two terms of \eqref{sol}, we then have
\begin{equation*}
\bigg\|\cos(t\sqrt{1-\Delta})f + \frac{\sin(t\sqrt{1-\Delta})}{\sqrt{1-\Delta}}g\bigg\|_{L_t^q H_r^\sigma}
\lesssim \|f\|_{H^{1/2}} + \|g\|_{H^{-1/2}} 
\end{equation*}
under the same conditions on $(q,r)$ and $\sigma$ as in Theorem \ref{thm2}.
Now it remains to show that
\begin{align*}\label{Sinhomo}
\|\mathscr{S}u\|_{L^q_t H^\sigma_r}&= \bigg\|\int_{0}^{t} \frac{\sin((t-s)\sqrt{1-\Delta})}{\sqrt{1-\Delta}}\big(V(\cdot)u(\cdot,s)\big) ds\bigg\|_{L^q_t H^\sigma_r}\\
&\lesssim \|V\|_{\mathcal{F}^p} \big( \|f\|_{H^{1/2}} + \|g\|_{H^{-1/2}} \big)
\end{align*}
for the same $(q,r)$ and $\sigma$.
By duality, it suffices to show that
\begin{equation}\label{D-Str}
	\begin{split}
	\bigg<\langle \nabla \rangle^{\sigma}\int_{0}^{t} \frac{\sin((t-s)\sqrt{1-\Delta})}{\sqrt{1-\Delta}}&\big(V(\cdot)u(\cdot,s)\big) ds, G \bigg>_{x,t}\\
	&\lesssim \|V\|_{\mathcal{F}^p} \Big(\|f\|_{H^{1/2}} + \|g\|_{H^{-1/2}}\Big)\|G\|_{L_{t}^{q'}L_{x}^{r'}}.
	\end{split}
\end{equation}
The left-hand side of \eqref{D-Str} is equivalent to
\begin{align}\label{eq6}
\int_{-\infty}^{\infty}\int_{0}^{t}\Big< \langle \nabla \rangle^{\sigma-1}&\sin\big((t-s)\sqrt{1-\Delta}\big)\big(V(\cdot)u(\cdot,s)\big),G\Big>_{x}dsdt\nonumber\\
&=\int_{-\infty}^{\infty}\int_{0}^{t} \big< Vu, \langle \nabla \rangle^{\sigma-1} \sin\big((t-s)\sqrt{1-\Delta}\big)G\big>_{x}dsdt\nonumber\\
&=\bigg< V^{1/2}u, V^{1/2} \langle \nabla \rangle^{\sigma-1} \int_{s}^{\infty} \sin\big((t-s)\sqrt{1-\Delta}\big)G \,dt\bigg >_{x,s}.
\end{align}
Using H\"older's inequality, \eqref{eq6} is bounded by
\begin{equation*}
\|u\|_{L_{x,s}^{2} (|V|)} \bigg\| \langle \nabla \rangle^{\sigma-1}\int_{s}^{\infty} \sin\big((t-s)\sqrt{1-\Delta}\big)G \,dt\bigg\|_{L_{x,s}^{2}(|V|)}.
\end{equation*}
We will show that
\begin{equation}\label{B1-Str}
\|u\|_{L_{x,t}^{2}(|V|)} \lesssim \|V\|_{\mathcal{F}^p}^{1/2} \big(\|f\|_{H^{1/2}} + \|g\|_{H^{-1/2}}\big)
\end{equation}
and
\begin{equation}\label{B2-Str}
\bigg\| \langle \nabla \rangle^{\sigma-1}\int_{t}^{\infty} \sin\big((t-s)\sqrt{1-\Delta}\big)G\,ds \bigg\|_{L_{x,t}^{2}(|V|)} \lesssim \|V\|_{\mathcal{F}^p}^{1/2} \|G\|_{L_{t}^{q'} L_{x}^{r'}}.
\end{equation}
Then the desired estimate \eqref{D-Str} is proved.

To show the first estimate \eqref{B1-Str}, we apply \eqref{homo} and \eqref{WLinhomo} to \eqref{sol} to get
\begin{equation}\label{eq232}
\begin{split}
\|u\|_{L^2_{x,t}(|V|)} \lesssim \|V\|^{1/2}_{\mathcal{F}^p} \big(\|f\|_{H^{1/2}} + \|g\|_{H^{-1/2}}\big)+ \|V\|_{\mathcal{F}^p}\|u\|_{L^2_{x,t}(|V|)}.
\end{split}
\end{equation}
Since we are assuming that $\|V\|_{\mathcal{F}^p}$ is small enough, the last term on the right-hand side of \eqref{eq232}
can be absorbed into the left-hand side.
Hence, we get \eqref{B1-Str}.
For the second estimate \eqref{B2-Str}, we first note that
\begin{align*}
\Big\| \langle \nabla \rangle^{\sigma-1} \int_{-\infty}^{\infty} \sin\big((t-s)&\sqrt{1-\Delta}\big)G \,ds \Big\|_{L_{x,t}^{2}(|V|)}\\
&\lesssim \Big\| \langle \nabla \rangle^{\sigma-1} e^{it\sqrt{1-\Delta}} \int_{-\infty}^{\infty} e^{-is\sqrt{1-\Delta}}\,G \,ds \Big\|_{L_{x,t}^{2}(|V|)}\\
&\lesssim \|V\|_{\mathcal{F}^p}^{1/2}\Big\| \langle \nabla \rangle^{\sigma-1/2} \int_{-\infty}^{\infty} e^{-is\sqrt{1-\Delta}}\,G\, ds \Big\|_{L^2}\\
&\lesssim \|V\|_{\mathcal{F}^p}^{1/2} \| G \|_{L_{t}^{q'} L_{x}^{r'}}.
\end{align*}
using \eqref{homo} and the dual estimate of \eqref{Shomo}.
We then use the following Christ-Kiselev lemma (\cite{CK}) to conclude
\begin{equation*}
\Big\|\langle \nabla \rangle^{\sigma-1} \int_{-\infty}^{t} \sin\big((t-s)\sqrt{1-\Delta}\big) G\,ds \Big\|_{L_{x,t}^{2}(|V|)}
\lesssim \|V\|_{\mathcal{F}^{p}}^{1/2} \|G\|_{L_{t}^{q'}L_{x}^{r'}}
\end{equation*}
for $2>q'$, which yields \eqref{B2-Str} simply by changing some variables.

\begin{lem}[Christ-Kiselev lemma]\label{CK}
Let $X$ and $Y$ be two Banach spaces and let $T$ be a bounded linear operator from $L^{\alpha}(\mathbb{R};X)$ to $L^{\beta}(\mathbb{R};Y)$ such that
	\begin{equation*}
	Tf(t)=\int_{-\infty}^{\infty}K(t,s)f(s)\,ds.
	\end{equation*}
	Then the operator
	\begin{equation*}
	\widetilde{T}f(t)=\int_{-\infty}^{t}K(t,s)f(s)\,ds
	\end{equation*}
	has the same boundedness when $\beta > \alpha$, and $\|\widetilde{T}\| \lesssim \|T\|$.
\end{lem}

\

\noindent\textbf{Acknowledgment.}
The authors would like to thank the anonymous referees for careful reading of the manuscript
and some helpful comments.


\begin{thebibliography}{99}

\bibitem{AH} S. Agmon and L. H\"{o}rmander, \textit{Asymptotic properties of solutions of differential equations with simple characteristics},
J. Analyse Math. 30 (1976), 1-38.

\bibitem{BBCRV} J. A. Barcel\`o, J. M. Bennett, A. Carbery, A. Ruiz and M. C. Vilela, \textit{A note on weighted estimates for the Schr\"odinger operator}, Rev. Mat. Complut. 21 (2008), 481-488.

\bibitem{B} P. Brenner, \textit{On space-time means and everywhere defined scattering operators for nonlinear Klein-Gordon equations},
Math. Z. 186 (1984), 383-391.

\bibitem{CS} S. Chanillo and E. Sawyer, \textit{Unique continuation for $\Delta + v$ and the C. Fefferman-Phong class},
Trans. Amer. Math. Soc. 318 (1990), 275-300.

\bibitem{CR} F. Chiarenza and A. Ruiz, \textit{Uniform $L^2$-weighted Sobolev inequalities}, Proc. Amer. Math. Soc. 112 (1991), 53-64.

\bibitem{CK} M. Christ and A. Kiselev, \textit{Maximal functions associated to filtrations}, J. Funct. Anal. 179 (2001), 409-425.

\bibitem{D'A} P. D'Ancona, \textit{Kato smoothing and Strichartz estimates for wave equations with magnetic potentials},
Comm. Math. Phys. 335 (2015), 1-16.

\bibitem{D'A2} P. D'Ancona, \textit{On large potential perturbations of the Schr\"odinger, wave and Klein-Gordon equations},
Commun. Pure Appl. Anal. 19 (2020), 609-640.

\bibitem{D'AF} P. D'Ancona and L. Fanelli, \textit{Strichartz and smoothing estimates for dispersive equations with magnetic potentials},
Comm. Partial Differential Equations 33 (2008), 1082-1112.

\bibitem{K} T. Kato, \textit{Wave operators and similarity for some non-selfadjoint operators}, Math. Ann. 162 (1966), 258-279.

\bibitem{KY} T. Kato and K. Yajima, \textit{Some examples of smooth operators and the associated smoothing effect},
Rev. Math. Phys. 1 (1989), 481-496.

\bibitem{RS} M. Reed and B. Simon, \textit{Methods of modern mathematical physics, II: fourier analysis,
self-adjointness}, Academic Press, New York, 1975.

\bibitem{RV} A. Ruiz and L. Vega, \textit{Local regularity of solutions to wave equations with time-dependent potentials}, Duke Math. J. 76 (1994), 913-940.

\bibitem{S} I. Seo, \textit{From resolvent estimates to unique continuation for the Schr\"odinger equation},
Trans. Amer. Math. Soc. 368 (2016), 8755-8784.

\bibitem{So} C. D. Sogge, \textit{Fourier integrals in classical analysis}, Cambridge Tracts in Mathematics, 105. Cambridge University Press, Cambridge, 1993.

\bibitem{St} R. Strichartz, \textit{Restrictions of Fourier transforms to quadratic surfaces and decay of solutions of wave equations},
Duke Math. J. 44 (1977), 705-714.

\bibitem{ZZ} J. Zhang and J. Zheng,  \textit{Strichartz estimate and nonlinear Klein-Gordon equation on nontrapping scattering space},
J. Geom. Anal. 29 (2019), 2957-2984.

\end{thebibliography}
\end{document}